\documentclass[11pt]{amsart}
\usepackage{amsmath}
\usepackage{enumerate}
\usepackage{amssymb}
\usepackage{pb-diagram}
\usepackage{enumerate}
\usepackage{color}

\newtheorem{theorem}{Theorem}[section]
\newtheorem{thm}[theorem]{Theorem}
\newtheorem{prop}[theorem]{Proposition}

\newtheorem{fact}[theorem]{Fact}

\newtheorem{proposition}[theorem]{Proposition}
\newtheorem{lemma}[theorem]{Lemma}
\newtheorem{corollary}[theorem]{Corollary}

\theoremstyle{definition}

\newtheorem{qn}[theorem]{Question}

\theoremstyle{remark}

\newtheorem{remark}[theorem]{Remark}
\newtheorem{claim}{Claim}

\newcommand{\N}{\mathbb{N}}
\newcommand{\Q}{\mathbb{Q}}
\newcommand{\R}{\mathbb{R}}

\global\long\def\GL{\mathrm{GL}}

\newcommand{\h}{\mathfrak h}
\newcommand{\g}{\mathfrak g}

\newcommand{\oo}{^{\text{o}}}

\newcommand{\gl}{\mathfrak{gl}}

\begin{document}

\title{The derived subgroup of linear and simply-connected o-minimal groups}
\date{\today}

\author{El\'ias Baro}

\address{Departamento de
  \'Algebra; Facultad de Matem\'aticas; Universidad Complutense de
  Madrid; Plaza de Ciencias, 3; Ciudad Universitaria; 28040 Madrid;
  Spain}

\email{eliasbaro@pdi.ucm.es}

\thanks{Research partially supported by the program MTM2017-82105-P}

\begin{abstract}We show that the derived subgroup of a linear definable group in an o-minimal structure is also definable, extending the semialgebraic case proved in \cite{P}. We also show the definability of the derived subgroup in case that the group is simply-connected.
\end{abstract}

\maketitle
\section{Introduction}

Let $\mathcal{R}=\langle R,<, +,\cdot,\cdots\rangle$ be an o-minimal expansion of a real closed field $R$. Algebraic groups over $R$ are clearly definable in $\mathcal{R}$; on the other hand, if $G$ is a group definable in $\mathcal{R}$ and $R=\R$ then $G$ has a Lie group structure (see Preliminaries). In fact, the behaviour of groups definable over $\mathcal{R}$ --hereafter \emph{o-minimal groups}-- rests in between algebraic groups and Lie groups. The definability of the derived subgroup is a good example of this dichotomy. 

As it is well-known, the derived subgroup of an irreducible algebraic group is an irreducible algebraic subgroup. In the context of Lie groups, the derived subgroup is a virtual subgroup (i.e., the image of a Lie homomorphism). However, there are examples of Lie groups --even \emph{solvable}-- whose derived subgroup is not closed  \cite[Ex.1.4.4]{OV}. In two important situations it is closed: either if the Lie group is linear or it is simply-connected. In both cases the proof relies on Lie's third fundamental theorem and therefore it cannot be reproduced in the o-minimal setting.

A. Conversano \cite[\S 1]{CP} showed an example of an o-minimal group $G$ whose derived subgroup $G'$ is not definable (remarkably, the example is semialgebraic over $\R$). Thus, the situation concerning the derived subgroup of o-minimal groups could seem closer to  Lie groups rather than to the algebraic ones. However, Conversano's example is a central extension of a simple group and therefore it is not solvable. Surprisingly, in \cite{BJO} we proved that if $G$ is a solvable  connected  o-minimal group then $G'$ is definable. Moreover, the commutator width of $G$ is bounded by $\dim(G)$. Recall that the derived  (or commutator) subgroup of $G$ is
$$G'=\bigcup_{n\in \N} [G,G]_n$$ where $[G,G]_n$ denotes the definable set of at most $n$ products of commutators. The commutator width is the smallest $n\in \N$ such that $G'=[G,G]_n$ in case it exists.

\medskip
\emph{In this paper we prove that if $G$ is a connected o-minimal group and $G$ is either linear (Theorem \ref{linear}) or simply-connected (Theorem \ref{MainTheoConnected}) then $G'$ is definable}.

\medskip
In Section \ref{s:linear} we address the linear case. A. Pillay already showed in \cite{P} that if $G$ is semialgebraic and linear then $G'$ is semialgebraic. He avoids the use of Lie's third fundamental theorem by considering the Zariski closure.  We will combine the result of Pillay with the definability of the derived subgroup in the solvable case established in \cite{BJO}. Furthermore, we will also provide a bound of the commutator width.

In Section \ref{s:simplyconnected} we make use of the developments of o-minimal homotopy in \cite{BO} to show that a normal connected definable subgroup of a simply-connected definable group is simply-connected (Proposition \ref{propsimplycom}). This allows us to make induction arguments, so we can apply the strategy used in \cite{BJO} to reduce the problem to a minimal configuration: a central extension of a semisimple group (Proposition \ref{miconf}).

Finally, in Section \ref{s:malcev} we apply our results to prove an o-minimal version of a classical result by A. Malcev (Theorem \ref{miconf}) concerning the existence of cross-section of projection maps of quotients of simply-connected Lie groups.

The author wishes to thank the referee for her/his valuable comments and suggestions.

\section{Preliminaries}\label{preliminaries}

We fix an o-minimal expansion $\mathcal{R}$ of a real closed field $R$. Henceforth \emph{definable} means definable in the structure $\mathcal{R}$ possibly with parameters. Let $G$ be a definable group in $\mathcal{R}$, we refer to \cite{O}  for the basics on o-minimal groups. 

For any fixed $p\in \N$, the group $G$ is a topological group with a definable  $\mathcal{C}^p$-manifold structure compatible with the group operation. Any definable subgroup of $G$ is closed and a $\mathcal{C}^p$-submanifold of $G$. Since $\mathcal{R}$ expands a field, we have elimination of imaginaries and therefore the quotients of definable groups by definable subgroups are again definable.  A definable group $G$ is \emph{linear} if $G\leq \GL(n,R)$ for some $n\in \N$.

Any definable subset $X$ of $G$ is the disjoint union of finitely many connected definable components, i.e. definable subsets which cannot be written as the union of two proper open definable subsets. In particular, the connected component $G\oo$ of $G$ which contains the identity is a --normal-- subgroup of finite index (it is the smallest one with that property). We say that $G$ is \emph{connected} if $G\oo=G$. Moreover, the group $G$ has the {\em descending chain condition 
on definable subgroups} ($dcc$ for short): any strictly descending chain of definable subgroups --which must be closed in the topology-- of $G$ is finite. 

In \cite{PPS00} the authors  define the Lie algebra of $G$ similarly as in the classical case. We define the tangent space $\mathcal{T}_e(G)$ as the set of all equivalence classes of definable $C^1$-curves $\sigma:(-1,1)	\rightarrow G$ with $\sigma(0)=e$, where two curves are equivalent if they are tangent
at $0$. We endow $\mathcal{T}_e(G)$ with
the natural vector space structure as in the classical case. Given a local $C^3$-chart $\varphi:U\rightarrow R^n$ around $e\in U\subseteq G$, $\varphi(e)=0$, we can identify $\mathcal{T}_e(G)$ with $R^n$ via the isomorphism sending the equivalence class of a definable $C^1$-curve $\sigma$ to $(\varphi \circ \sigma)'(0)$. Moreover, since $\varphi(ey)=\varphi(y)$ and $\varphi(xe)=\varphi(x)$, using the Taylor expansion we get that
$$\varphi(xy)=\varphi(x)+\varphi(y)+\alpha(\varphi(x),\varphi(y))+\cdots$$
where $\alpha$ is a bilinear vector-valued form and dots stands for elements of order greater than $2$.
The transposition of $x$ and $y$ yields
$$\varphi(yx)=\varphi(y)+\varphi(x)+\alpha(\varphi(y),\varphi(x))+\cdots$$
and therefore we get that
\begin{equation}\label{e:comm}\varphi([x,y])=\varphi(x^{-1}y^{-1}xy)=\gamma(\varphi(x),\varphi(y))+\cdots 
\end{equation}
where $\gamma(\varphi(x),\varphi(y))=\alpha(\varphi(x),\varphi(y))-\alpha(\varphi(y),\varphi(x))$ and dots stand for the terms of order greater then $2$. It turns out that $\mathcal{T}_e(G)$ with the bracket operation $[X,Y]:=\gamma(X,Y)$ is the \emph{Lie algebra} of $G$, denoted by $\text{Lie}(G)$ --and which do not depend on the chart $\varphi$ chosen.

Many basic results from the Lie theory have an o-minimal analogue. For example, if $H$ is a definable subgroup of $G$ then $\text{Lie}(H)$ is a Lie subalgebra of $\text{Lie}(G)$. Furthermore:

\begin{fact}\cite[Claim 2.20]{PPS00}\label{LieAlgChar} If $G_1$ and $G_2$ are two connected definable subgroups of  a definable group $G$ with the same Lie algebra then $G_1=G_2$. 
\end{fact}

In some sense, locally definable groups play the role of virtual Lie groups. A \emph{locally definable group} \cite{E05} is a subset $\mathcal{G}=\bigcup_{n\in \N}X_n$ of $R^\ell$ which is a countable union  of increasing definable subsets $X_n$ and whose group operation restricted to $X_n\times X_n$ is contained in $X_m$ for some $m\in \N$ and it is a definable map. A homomorphism $f:\mathcal{G}\rightarrow \mathcal{H}$ of locally definable groups $\mathcal{G}=\bigcup_{n\in \N}X_n$ and $\mathcal{H}=\bigcup_{m\in \N}Y_m$ is a \emph{locally definable homomorphism} if for each $n$ there is $m$ such that $f(X_n)\subseteq Y_m$ and the restriction $f\upharpoonright X_n$ is definable. For instance, a dense spiral around a real torus is an example of a locally definable subgroup of a definable group (and note that if we move to a saturated model then the spiral becomes closed in the torus, so that we can have locally definable closed subgroups of definable groups).

As before, $\mathcal{G}$ has a locally definable $\mathcal{C}^p$-manifold structure, a submanifold of $\GL(n,R)$ in case that the group $\mathcal{G}$ is linear. We say that a subset $Y\subseteq \mathcal{G}$ is \emph{compatible} if $Y\cap X_n$ is definable for each $n\in \N$. A compatible subset $Y\subseteq \mathcal{G}$ is \emph{connected} if it cannot be 
written as a union of two proper open compatible subsets; any compatible subset is a disjoint union of its countable many (clopen) compatible connected components (see \cite[\S 4]{BOldh}).

\begin{fact}\label{Derlocally}Let $G$ be a connected definable group and let $\g$ be its Lie  algebra. Then $G'$ is a connected locally definable subgroup such that $[\g,\g]\subseteq \text{Lie}(G')$.
\end{fact}
\begin{proof}Since $G\times G\rightarrow G:(x,y)\mapsto [x,y]=x^{-1}y^{-1}xy$ is continuous we get that $[G,G]_1$ is a connected definable subset of $G$. Therefore $[G,G]_n=[G,G]_1\overset{n}{\cdots} [G,G]_1$ is connected and definable for each $n\in \N$. In particular, $G'=\bigcup_{n\in \N}[G,G]_n$ is locally definable because $[G,G]_n[G,G]_n=[G,G]_{2n}$ for each $n\in \N$. On the other hand, $[G']\oo\cap [G,G]_n$ is an open and closed definable subset of $[G,G]_n$ and therefore $[G,G]_n\subseteq [G']\oo$, so that $[G']\oo=G'$.

Next, let us see that $[\g,\g]\subseteq \text{Lie}(G')$, we follow \cite[\S 4]{OV}. It suffices to show that given $X,Y\in \g$ then $[X,Y]\in \text{Lie}(G')$. Let $x,y:(-1,1)\rightarrow G$ be definable $C^1$-curves such that $x'(0)=X$ and $y'(0)=Y$. Then 
$$c(t):=\left\{\begin{array}{l}
[x(\sqrt{t}),y(\sqrt{t})], \text{ if } t\in [0,1),\\
{[x(\sqrt{|t|}),y(\sqrt{|t|})]}^{-1}, \text{ if } t\in (-1,0],
\end{array}\right.$$
is a definable $C^1$-curve such that $c'(0)=[X,Y]$. Indeed, by equation (\ref{e:comm}) above,
$$\overline{[x(\sqrt{t}),y(\sqrt{t})]}=\gamma(\overline{x(\sqrt{t})},\overline{y(\sqrt{t})})+\cdots$$
where ``$\overline{\phantom{H}}$" denotes the image by the chart $\varphi$. Thus, 
$$\lim_{t\rightarrow 0^{+}} \tfrac{1}{t} \overline{[x(\sqrt{t}),y(\sqrt{t})]}=\lim_{t\rightarrow 0^{+}} \tfrac{1}
{t}\gamma(\overline{x(\sqrt{t})},\overline{y(\sqrt{t})})$$
$$=\lim_{t\rightarrow 0^{+}}\gamma(\overline{\tfrac{x(\sqrt{t})}{\sqrt{t}}},	\overline{\tfrac{y(\sqrt{t})}{\sqrt{t}}})=\gamma(X,Y)=[X,Y].$$
On the other hand,
$$\lim_{t\rightarrow 0^{-}} \tfrac{1}{t} \overline{[x(\sqrt{|t|}),y(\sqrt{|t|})]}=\lim_{t\rightarrow 0^{-}} \tfrac{1}
{t}\gamma(\overline{x(\sqrt{-t})},\overline{y(\sqrt{-t})})$$
$$=\lim_{t\rightarrow 0^{-}}-\gamma(\overline{\tfrac{x(\sqrt{-t})}{\sqrt{-t}}},	\overline{\tfrac{y(\sqrt{-t})}{\sqrt{-t}}})=-\gamma(X,Y)=-[X,Y]$$
and therefore $\lim_{t\rightarrow 0^{-}} \tfrac{1}{t} \overline{[x(\sqrt{|t|}),y(\sqrt{|t|})]^{-1}}=[X,Y]$, as required.\end{proof}

\begin{remark}Back to the analytic setting, if $G$ is a connected Lie group then $G'$ is a virtual Lie subgroup whose Lie algebra is $[\g,\g]$. The proof uses again in a crucial way Lie's third fundamental theorem, see \cite[Thm.\,1]{OV}. For our purpose concerning the definability of the derived subgroup of certain o-minimal groups, the inclusion provided by Fact \ref{Derlocally} above is enough. However, it is natural to ask if the Lie algebra of the locally definable subgroup $G'$ of a connected definable group $G$ is $[\g,\g]$. We would like to point out that the answer would be positive in case that  the answer of Question \ref{conjecture} is affirmative. For, if $p:\widetilde{G}\rightarrow G$ denotes the (locally definable) universal covering of the connected definable group $G$, then the Lie algebra of $\widetilde{G}$ is again $\g$. By the conjecture we would have that $\widetilde{G}'$ is a compatible subgroup of $\widetilde{G}$ whose Lie algebra is $[\g,\g]$. In particular,  $p\upharpoonright_{\widetilde{G}'}:\widetilde{G}'\rightarrow G'$ is the universal covering of the locally definable group $G'$ and therefore $\text{Lie}(G')=\text{Lie}(\widetilde{G}')=[\g,\g]$, as required.
\end{remark}

A Lie subalgebra  of $\gl(n,R)$  is said to be \emph{algebraic} if it is the Lie algebra of an algebraic subgroup of $\GL(n,R)$. Given  a Lie subalgebra $\g$ of $\gl(n,R)$, $a(\g)$ denotes the  minimal algebraic Lie subalgebra of $\gl(n,R)$ containing $\g$. We recall that if $\mathfrak{g}$ is a subalgebra of $\gl(n,R)$ then $[\mathfrak{g},\mathfrak{g}]=[a(\g),a(\g)]$ is algebraic (see \cite[Ch.3, \S 3]{OV}). 

If $G$ is a semialgebraic subgroup of $\GL(n,R)$ then $G$ and its Zariski closure $\overline{G}$ in $\GL(n,R)$ have the same dimension. This is a crucial aspect in the proof of the following result (the bound in the commutator width can be deduced from the proof there noting that $[G,G]_n^{-1}=[G,G]_n$ for each $n\in \N$):

\begin{fact}\cite[Cor.3.3]{P}\label{Pillay} Let $G$ be a semialgebraic subgroup of $\GL(n,R)$. Then $G'$ is semialgebraic and its commutator width is bounded by $\dim(G)$. 
\end{fact}

It is no longer true in general that if $G$ is a linear o-minimal group then $\dim(G)=\dim(\overline{G})$. Nevertheless, in the proof of Theorem \ref{linear} we will make use of the fact that $G$ and $\overline{G}$ still have a strong relation precisely because the Lie algebra of $G'$ and $\overline{G}'$ coincide (as essentially pointed out in the proof of \cite[Thm.\,4.1]{PS}). We recall that  by \cite[Lem.2.4]{PS} the Zariski Lie algebra of an algebraic subgroup of $\GL(n,R)$ and its o-minimal Lie algebra canonically coincides. Moreover:

\begin{fact} \cite[Prop.3.9]{BBO}\label{zariskiclo} Let $G$ be a definable subgroup of $\GL(n,R)$ and $\overline{G}$ its
Zariski closure. Then, $\text{\emph{Lie}}(\overline{G}) = a(\text{\emph{Lie}}(G))$. Furthermore, if $G$ is connected then $\overline{G}$ is irreducible and $G$ is normal in $\overline{G}$.
\end{fact}

The \emph{solvable radical} $R(G)$ of a definable group $G$ is the maximal normal solvable connected definable subgroup of $G$. We say that $G$ is \emph{semisimple} if $R(G)$ is trivial. A non-abelian definable group $G$ is \emph{definably simple} if it has no proper non-trivial normal definable subgroup. We recall a basic result of semisimple groups:

\begin{fact}\cite{HPP}\label{semisimple} Let $G$ be a connected definable group. Then $G$ is semisimple if and only if its Lie algebra $\g$ is semisimple. In this case, we have that $G'=G$, the center $Z(G)$ is finite, and $G/Z(G)$ is definably isomorphic to a direct product of finitely many definably simple groups.
\end{fact}

We finish with some topological remarks on o-minimal groups. Given a connected definable group $G$ in the o-minimal structure $\mathcal{R}$, we define the \emph{o-minimal $n$-homotopy group} $\pi_n^{\mathcal{R}}(G)$ as in the classical case via definable maps and definable homotopies pointed in the identity element \cite[\S 4]{BO}. We say that $G$ is \emph{simply-connected} if $\pi^{\mathcal{R}}_1(G)=1$.

If $\mathcal{R}_1$ is an elementary extension of $\mathcal{R}$ and $G(R_1)$ is the realization of $G$ in $\mathcal{R}_1$, then $\pi^{\mathcal{R}_1}_n(G(R_1))$ and $\pi_n^{\mathcal{R}}(G)$ are canonically isomorphic, so henceforth we shall omit the superscript and we will write $\pi_n(G)$. For, it can be deduced from the following stronger result which will be crucial in our work:

\begin{fact}\cite[Thm.3.1, Cor.4.4]{BO}\label{homotopy} Let $X\subseteq R^n$ and $Y\subseteq R^m$ be connected semialgebraic sets defined over $\Q$. Then every continuous map $f:X\rightarrow Y$ definable in $\mathcal{R}$ is definably homotopic to a semialgebraic map $g:X\rightarrow Y$ defined over $\Q$. If $g_1,g_2:X\rightarrow Y$ are two continuous semialgebraic maps defined over $\Q$ which are definably homotopic, then they are semialgebraically homotopic over $\Q$. 

In particular, the o-minimal $n$-homotopy group $\pi_n(X)$ is canonically isomorphic to the classical homotopy group $\pi_n(X(\R))$ of the realization of $X$ in the real numbers.
\end{fact}

Similarly, we can define the o-minimal $n$-homotopy group of a locally definable group $\mathcal{G}$ and again we have invariance under elementary extensions. As it happens with Lie groups, the fundamental group interacts with map coverings: an onto locally definable homomorphism $p:\mathcal{G}\rightarrow \mathcal{H}$ is  a \emph{locally definable covering} if there is a family of open definable subsets $\{U_j\}_{j\in J}$ of $\mathcal{H}=\bigcup_{n\in \N}H_n$ whose union is $\mathcal{H}$, each $H_n$ is contained in the union of finitely many $U_j$, and each $p^{-1}(U_j)$ is a disjoint union of open definable subsets of $\mathcal{G}$ each of which is mapped homeomorphically onto $U_j$.

\begin{fact}\label{nocover}Let $\mathcal{G}$ and $\mathcal{H}$ be a connected and a simply-connected locally definable group respectively, and let $f:G\rightarrow H$ be a locally definable surjective homomorphism. If $dim(\text{ker}(f))=0$ then $f$ is an isomorphism.
\end{fact}
\begin{proof}Since $\dim(\text{ker}(f))=0$, the map $f$ is a locally definable covering \cite[Thm.3.6]{E05}. Therefore, by \cite[Prop.3.4 and 3.12]{E05},  
$$\text{ker}(f)=\pi_1(\mathcal{H})/f_*(\pi_1(\mathcal{G})),$$ where 
$f_*:\pi_1(G)\rightarrow \pi_1(H):[\gamma] \mapsto [f\circ \gamma]$. Since $\pi_1(\mathcal{H})=1$ we get that $\text{ker}(f)=1$, as desired. \end{proof}

\section{Linear groups}\label{s:linear}

If $G$ is a connected linear Lie subgroup of $\GL(n,\R)$ then $G'$ is a closed subgroup. Indeed, if $\g$ denotes the Lie algebra of $G$ then we have that $G'$ is a connected virtual subgroup of $G$ whose Lie algebra is $[\g,\g]$. On the other hand, $[\g,\g]$ is the Lie algebra of an algebraic subgroup $H$ of $\GL(n,R)$. Therefore,  $G'$ and $H\oo$ have the same Lie algebra, so they are equal (see Ch.1 \S 2 and Ch.4 \S1 in \cite{OV}). It follows that $G'$ is closed in $G$.

We cannot adapt the above argument to prove that if $G$ is a linear o-minimal group then $G'$ is a definable subgroup. Though $G'$ is a connected locally definable subgroup of $G$, it is not true that connected locally definable subgroups of $G$ are uniquely determined by their Lie algebra. For example, the group $R$ and its finite elements Fin$(R)$ have the same Lie algebra.
\begin{thm}\label{linear}Let $G\leq\GL(n,R)$ be a connected definable group in $\mathcal{R}$. Then $G'$ is connected semialgebraic subgroup of $G$ whose Lie algebra is $[\g,\g]$. Moreover, the commutator width of $G$ is bounded by $\dim(G)+\dim(G')-\dim(G'')$.
\end{thm}
\begin{proof}Let $\mathfrak{g}$ be the Lie algebra of $G$. By Fact \ref{zariskiclo} the Zariski closure $H:=\overline{G}\leq \GL(n,R)$ of $G$ is an irreducible algebraic subgroup of $\GL(n,R)$ whose Lie algebra $\h:=\text{Lie}(H)$ equals $a(\g)$. The derived subgroup $H'$ of $H$, is also an irreducible algebraic group with $\text{Lie}(H')=[a(\g),a(\g)]=[\g,\g]$, see \cite[Prop.\,7.8]{Bo}.  Denote $G_1:=H\oo$ and $G_2:=[H']\oo$, which are connected semialgebraic subgroups of $\GL(n,R)$. Since $\text{Lie}(G_2)=[\g,\g] \subseteq \g$ it follows from Fact \ref{LieAlgChar} that
$$G_2 \trianglelefteq G \trianglelefteq G_1.$$
We prove that $G'$ equals the connected semialgebraic group $G_2$. By Fact \ref{Pillay} the groups $G'_1$ and $G'_2$ are both semialgebraic and connected. Thus, the quotient $G_1/G_2$ is abelian since $G'_1=[G_1,G_1]=[H\oo,H\oo]\trianglelefteq [H,H]\oo=G_2$. In particular $G/G_2$ is abelian, so that $G'\leq G_2$.

On the other hand, consider the connected definable group $G/G'_2$ and note that it is non-necessarily linear. However, it is solvable. Indeed, we already showed above that $G'_1 \trianglelefteq G_2$, so that 
$[G_1/G'_2]'=G'_1/G'_2\trianglelefteq G_2/G'_2$ is abelian. Then $[G/G'_2]'\trianglelefteq [G_1/G'_2]'$ is abelian and therefore $G/G'_2$ is solvable, as desired. Thus, by \cite[Thm.3.1]{BJO} we deduce that $[G/G'_2]'=G'/G'_2$ is definable and connected, and the commutator width of $G/G'_2$ is bounded by $\dim(G/G'_2)$. In particular, the equivalence classes of the definable quotient $G'/G'_2$ form a definable family of subsets of $G$ and therefore its union $G'$ is definable, as required. Moreover, since both $G'_2$ and $G'/G'_2$ are connected, $G'$ is also connected.

Finally, by Fact \ref{Derlocally} we have that 
$$\text{Lie}(G_2)=\text{Lie}([H']\oo)=\text{Lie}(H')=[\g,\g]\subseteq \text{Lie}(G')$$
and therefore $G'=G_2$ by Fact \ref{LieAlgChar}. In particular, $\text{Lie}(G')=\text{Lie}(G_2)=[\g,\g]$. Note that $G''=G'_2$ and thus by the above we have that the commutator width of $G/G''$ is bounded by $\dim(G/G'')=\dim(G)-\dim(G'')$.  Hence, the commutator width of $G$ is bounded by $\dim(G)-\dim(G'')+\dim(G')$. Indeed, for any $x\in G'$ there are $y_1,\ldots,y_\ell\in [G,G]_1$ such that
$$y_1^{-1}\cdots y_\ell^{-1}x\in G''$$ where  $\ell:=\dim(G)-\dim(G'')$. On the other hand, by Fact \ref{Pillay} the commutator width of $G'=G_2$ is bounded by $m:=\dim(G')$, so there are $z_1,\ldots,z_m \in  [G',G']_1\subseteq [G,G]_1$ such that $y_1^{-1}\cdots y_\ell^{-1}x=z_1\cdots z_m$ and therefore $x=y_1\cdots y_\ell z_1 \cdots z_m\in [G,G]_{\ell+m}$, as required.\end{proof}

We complete the linear case by considering non-connected definable linear groups:
\begin{corollary}Let $G$ be a linear group definable in an o-minimal structure, and $A$ and
$B$ be two definable subgroups which normalize each other. Then the subgroup $[A,B]$ is definable and $[A, B]\oo=[A\oo,B][A,B\oo]$. Furthermore, any element of $[A,B]\oo$ can be expressed as the product of at most $dim([A,B]\oo)$ commutators from $[A\oo,B]$ or $[A,B\oo]$ whenever $A\oo$ or
$B\oo$ is solvable. 
\end{corollary}

\begin{proof}It suffices to prove that $H:=A\oo B\oo$ satisfies condition $(*)$ of \cite[Thm.3.1]{BJO}. That is, if $K$ is a normal definable subgroup of $H$ such that $H/K$ is the central extension of a definable simple group then $(H/K)'=H'K/K$ is definable. Since $H$ is linear and connected, the latter follows from Theorem \ref{linear}.
\end{proof}

\section{Simply-connected definable groups}\label{s:simplyconnected}

A. Malcev proved the existence of cross-sections of quotients of simply-connected Lie groups by normal closed subgroups. This is a key result that, for example, it allows to study central extensions of simply-connected Lie groups via analytic sections \cite{H}.
\begin{fact}\cite{M}\label{Malcev} Let $G$ be a simply-connected Lie group and let $H\trianglelefteq G$ be a closed connected subgroup. Let $\pi:G\rightarrow G/H$ be the natural homomorphism. Then there exists an analytic mapping $\sigma:G/H\rightarrow H$ such that $\pi \circ \sigma=id$. 
\end{fact}

Note that with the above notation,
$$\begin{array}{lcr}
G\rightarrow (G/H)\times H\\
x \mapsto (\pi(x), x^{-1}\sigma(\pi(x)))
\end{array}
$$
is a homeomorphism and therefore both $H$ and $G/H$ are simply-connected. We are interested in an o-minimal version of this consequence because it will allow us to make arguments by induction. However, the proof in  \cite{M} goes through the $1$-$1$ correspondence between Lie algebras and simply-connected Lie groups, which it is not available in the o-minimal context.

We follow another approach. E. Cartan proved in \cite{Cartan} that any connected Lie group has trivial second homotopy group. His proof again goes through Lie's third fundamental theorem. W. Browder \cite{Br} later gave an alternative proof --using just homological methods-- which is also valid for H-spaces with finitely generated homology. Recall that a topological space $X$ is an \emph{H-space} if there exists a continuous map $f:X\times X \rightarrow X$ and an element $e\in X$ such that both $f(-,e)$ and $f(e,-)$ are homotopic to the identity map $\text{id}: X\mapsto X$. 

\begin{lemma}\label{pi2}Let $G$ be a connected definable group $G$. Then $\pi_2(G)=0$. 
\end{lemma}
\begin{proof}By the Triangulation theorem we can assume that there is a finite simplicial complex $K$ with vertices over $\Q$ such that $G=|K|(R)$, where $|K|(R)$ denotes the realization of $K$ in $R$. Moreover, we can assume that the identity of $G$ is one of the vertices.

By Fact \ref{homotopy}, the group operation on $|K|(R)$ is definably homotopic to a continuous semialgebraic map $f:|K|(R)\times |K|(R)\rightarrow |K|(R)$ which is defined over $\Q$. Furthermore, both $f(-,e)$ and $f(-,e)$ are clearly definably homotopic to the identity map $\text{id}$, so again by Fact \ref{homotopy} both are also semialgebraically homotopic to the $\text{id}$ over $\Q$. Thus, we can consider $f^{\R}:|K|(\R)\times |K|(\R)\rightarrow |K|(\R)$, the realization of $K$ and $f$ over the real numbers. The polyhedron $|K|(\R)$ with the map $f^{\R}$ is an H-space. Moreover, since $K$ is a finite simplicial complex the homology groups of $|K|(\R)$ are clearly finitely generated. Thus, by \cite[Thm.6.11]{Br} we have that $\pi_2(|K|(\R))=0$ and in particular $\pi_2(G)=0$, as required.
\end{proof}

A continuous definable map $p:E\rightarrow B$ is a \emph{definable fibration} if $p$ has the homotopy lifting property
with respect to all definable sets, i.e. for every definable set $X$, for every definable
homotopy $H: X \times I \rightarrow B$ and for every definable map $g : X \rightarrow E$ such that
$p\circ g=H(-,0)$ there is a definable homotopy $H_1 : X \times I \rightarrow  E$ such that $p\circ H_1=H$ and $H_1(-, 0)=g(-)$.

\noindent With the above lemma and the fact that the projection map of quotients of definable groups are definable fibrations we get:

\begin{proposition}\label{propsimplycom}Let $G$ be a connected definable group, and let $H$ be a normal connected definable subgroup of $G$. Then $G$ is simply-connected if and only if both $H$ and $G/H$ are simply-connected.
\end{proposition}
\begin{proof}By \cite[Cor.2.4]{BMO} the projection map $G\rightarrow G/H$ is a definable fibration. Therefore, by \cite[Thm.4.9]{BO}, for each $n\geq 2$, the o-minimal homotopy groups $\pi_n(G,H)$ and $\pi_n(G/H)$ are
isomorphic. In particular, we have the following long exact sequence via the o-minimal homotopy sequence of the pair $(G,H)$, see \cite[\S 4]{BO},
$$\pi_{2}(G/H)\rightarrow \pi_1(H) \rightarrow \pi_1(G)\rightarrow \pi_1(G/H)\rightarrow 0.$$
Since by Lemma \ref{pi2} we have that $\pi_2(G/H)=0$, we obtain the exact sequence
$$0\rightarrow \pi_1(H) \rightarrow \pi_1(G)\rightarrow \pi_1(G/H)\rightarrow 0.$$
Therefore,  $\pi_1(G)=0$ if and only if both $\pi_{1}(H)=0$ and $\pi_{1}(G/H)=0$, as required.
\end{proof}

Once we have that normal connected definable subgroups and their quotients are also simply-connected, we will be able to make induction arguments. For example, we have the following consequence analogue to the classical one:

\begin{corollary}\label{Cor1}Let $G$ be a connected solvable definable group. Then the following are equivalent:
\begin{enumerate}
 \item $G$ is torsion-free.
 \item $G$ is definably diffeomorphic to $R^{\, \dim(G)}$.
 \item $G$ is simply-connected. 
\end{enumerate} 
In particular, if $G$ is simply-connected then any connected definable subgroup is simply-connected.
\end{corollary}
\begin{proof} 1) implies 2) follows from \cite[Cor.\,5.7]{PS05} and 2) implies 3) is obvious. Let us prove by induction on the dimension that 3) implies 1), the initial case is obvious. Since $G$ is solvable, by \cite[Thm.\,4.1]{BJO} we have that $G'$ is a normal connected definable proper subgroup of $G$. If $G'\neq 1$ then by Proposition \ref{propsimplycom} and by the induction hypothesis we get that both $G'$ and $G/G'$ are torsion-free. In particular, $G$ is also torsion-free. 

Thus, we can assume that $G$ is abelian. We consider the definable homomorphism $f_n:G\rightarrow G:g\mapsto g^n$ for each $n\in \N$. By \cite[Cor.\,4.5]{BEM} we have that $\ker(f_n)\oo=1$ for all $n\in \N$. Thus, by Fact \ref{nocover}, we also have that $\ker(f_n)=1$ for all $n\in \N$, as required.

Finally, suppose that $G$ is simply-connected, and let $H$ be a connected definable subgroup of $G$. By the above equivalences, we have that $G$ is torsion-free, so that $H$ is also a connected torsion-free solvable definable group. Thus, $H$ is simply-connected, as required.\end{proof}

In order to prove the definability of the commutator subgroup of simply-connected groups we will be concerned with the following configuration: a definable central extension of a semisimple definable group. These extensions were profoundly studied in \cite{HPP}:

\begin{fact}\cite[Cor.5.3]{HPP}\label{f:HPP} Let $G$ be a connected central extension of a semisimple definable group. Then for each $n$, the set $Z(G)\cap [G,G]_n$ is finite. 
\end{fact}

\begin{prop}\label{miconf}Let $G$ be a simply-connected definable group such that $R(G)=Z(G)^\text{\emph{o}}$. Then $G'$ is definable and simply-connected. 
\end{prop}
\begin{proof}
By Fact \ref{Derlocally} the derived subgroup $G'$ of $G$ is a connected locally definable group of $G$ and the projection map 
$$\pi\upharpoonright_{G'}:G'\rightarrow G/Z(G)\oo$$
is a locally definable homomorphism because the restriction of $\pi$ to each $[G,G]_n$ is clearly definable. Since $R(G)=Z(G)\oo$, the connected definable group $G/Z(G)\oo$ is semisimple and therefore $[G/Z(G)\oo]'=G/Z(G)\oo$ by Fact \ref{semisimple}, so that $\pi\upharpoonright_{G'}$ is surjective. 

Moreover, by Fact \ref{f:HPP} the compatible subgroup $\text{ker}(\pi\upharpoonright_{G'})=G'\cap Z(G)\oo$ of $G'$ has dimension $0$. Therefore, since $G/Z(G)\oo$ is simply-connected by Proposition \ref{propsimplycom},  it follows from  Fact \ref{nocover}  that $\pi\upharpoonright_{G'}$ is a definable isomorphism, as required.
\end{proof}

We already have all the ingredients to prove the definability of the derived subgroup of simply-connected definable groups. 

\begin{fact}\cite[Cor.4.3]{BJO}\label{rudimentary} Let $G$ be a definable group and let $A,B$ be normal connected definable subgroups of $G$ with $[A,B]\leq Z(B)$ or $[A,B]\leq Z(A)$. Then $[A,B]$ is a connected definable subgroup of $G$.
\end{fact}
\begin{proof}For the sake of the presentation, we include a proof in case that $[A,B]\leq Z(B)$, the other one is similar. For any $a\in A$ and $b_1,b_2\in B$  we have $[a,b_1b_2]=[a,b_2][a,b_1]^{b_2}=[a,b_2][a,b_1]=[a,b_1][a,b_2]$. Thus, the set $[a,B]_1$ is a group, which is also definable and connected since it is the image of the continuous homomorphism $B\rightarrow B:b\mapsto [a,b]$. In particular, for any $a_1,\ldots,a_\ell\in A$ we have that
$$[a_1,B]_1\cdots[a_\ell,B]_1$$
is a connected definable subgroup of $B$. Therefore $[A,B]$ equals any such finite product of maximal dimension, so it is definable and connected.
\end{proof}

\begin{thm}\label{MainTheoConnected}
Let $G$ be a simply-connected group definable in an o-minimal structure, 
$A$ and $B$ be two normal connected definable subgroups of $G$. 
Then $[A,B]$ is a normal connected definable subgroup of $G$.
\end{thm}

\begin{proof}Let $G$ be a potential counterexample to our statement of minimal dimension. Note that $\dim(G)>2$ because otherwise $G$ is abelian. 
Let $A$ and $B$ be two normal connected definable subgroups of $G$ for which $[A,B]$ is either non-definable or definable but non-connected, and with
$$d:=\min(\dim(A),\dim(B))\geq 1$$ 
minimal. By Proposition \ref{propsimplycom} the normal definable subgroup $AB$ of $G$ is simply-connected and therefore $G=AB$. Note that since $A$ and $B$ are normal in $G$ we have that $[A,B]\trianglelefteq {A\cap B}$.

\begin{claim}\label{ClaimI} There is not a normal connected definable subgroup $C$ of $G$ contained in $A\cap B$ with
$\dim(C)<d$ and $C\nleq{Z(A)\cap Z(B)}$.
\end{claim}
\begin{proof}Suppose there exists such a subgroup $C$, and say it does not centralize $B$. Since $\dim(C)<d$ it follows that $[C,B]$ is a normal connected definable subgroup of $G$. Moreover, $[C,B]$ is non-trivial because $C\nleq Z(B)$.

Notice that $[C,B]$ is normal in $G$ and therefore by Proposition \ref{propsimplycom} we have that $G/[C,B]$  is simply-connected.
Denote by ``$\overline{\phantom{H}}$" the quotients by $[C,B]$. Since $G$ was the minimal counterexample, we get that 
$[\overline{A},\overline{B}]$ is definable and connected. But clearly
$[\overline{A},\overline{B}]=\overline{[A,B]}=[A,B]/[C,B]$ 
since $[C,B]\leq [A,B]$, and it follows that $[A,B]$ is definable and connected, a contradiction.
\end{proof}

\begin{claim}\label{Claim2}
We may assume $A={A\cap B}\nleq {Z(A)\cap Z(B)}$. 
\end{claim}
\begin{proof}If ${(A\cap B)}\leq{Z(A)\cap Z(B)}$ then $[A,B]\leq{Z(A)\cap
Z(B)}$. By Fact \ref{rudimentary} we obtain that  $[A,B]$ is definable and connected, a contradiction. Hence, we have that ${A\cap B}\nleq {Z(A)\cap Z(B)}$. 

On the other hand, by Claim \ref{ClaimI} we get that 
$\dim(A\cap B)\oo=\dim(A\cap B)=d$. 
Since $A$ and $B$ are definably connected it follows that $A\cap B$ equals
$A$ or $B$, say $A$. 
\end{proof}

In particular, we are now in the situation in which $A\trianglelefteq B=G$. 

\begin{claim}
\label{ClaimWMAA=B}
The subgroups $A$ and $B$ are equal. 
\end{claim}
\begin{proof}Suppose that $\dim(A)<\dim(B)=\dim(G)$. Then by minimality of our counterexample we have that $A'=[A,A]$ is a connected definable subgroup of $A$.

Since $A'$ is characteristic in $A$, and $A$ is normal in $B$, we get that $A'$ is normal in $B$. Thus, we can work in $B/A'$. We denote by ``$\overline{\phantom{H}}$" the quotients by $A'$. Note that $\overline{A}$ is abelian.
Then $[\overline{A},\overline{B}]\leq \overline{A}=Z(\overline{A})$ and Fact \ref{rudimentary} gives that 
$[\overline{A},\overline{B}]$ is connected and definable. Since $A'\leq [A,B]$, we deduce the definability
and connectedness of $[A,B]$, a contradiction. \end{proof}

All in all, we are in the following situation: 

\smallskip
\noindent\emph{$G$ is a simply-connected definable group for which $G'$ is either non-definable or definable but non-connected, and such that
any proper normal connected definable subgroup $C$ of $G$ is central in $G$.}

\medskip
The group is non-solvable, otherwise by \cite{BJO} we would have that $G'$ is definable and connected. Since $R(G)$ is a proper connected definable subgroup of $G$, we get that $R(G)\leq Z(G)$ and therefore $R(G)=Z(G)\oo$. Then, by Proposition \ref{miconf} it follows that $G'$ is definable and connected, a contradiction.
\end{proof}

Natural examples of simply-connected o-minimal groups appear in the literature: the spin groups or the examples in \cite[\S 1]{PS} of solvable o-minimal groups which are not semialgebraic. However,  we would like to stress that simply-connectedness emerges canonically in the context of locally definable groups. Indeed, every o-minimal group has a (simply-connected) universal cover which is a locally definable group. Hence, it seems natural to ask:
\begin{qn}\label{conjecture}Let $G$ be a locally definable group $G$ which is the universal covering of a connected o-minimal group. Is $G'$ a simply-connected compatible subgroup of $G$ whose Lie algebra is $[\g,\g]$? 
\end{qn}

\begin{remark}We would like to finish this section by pointing out that recently we notice that part of the results in \cite{BJO} can be generalized to an  abstract model-theoretic context. Let $G$ be a group interpretable in a structure $\mathcal{M}$. Henceforth, definability refers to $\mathcal{M}^{eq}$. We suppose that to each definable set in Cartesian
powers of $G$ is attached a dimension in $\N$, denoted by $\dim$ and satisfying the following axioms:

\smallskip
\noindent(\emph{Definability})
If $f$ is a definable function between two definable sets $A$ and $B$, then for
every $m$ in $\N$ 
the set $\{ b\in B~|~\dim(f^{-1}(b))=m \}$ is a definable subset of $B$. \\
(\emph{Additivity})
If $f$ is a surjective definable function between two definable sets $A$ and $B$, whose
fibers have constant 
dimension $m$ in $\N$, then $\dim(A)=\dim(B)+m$. \\
(\emph{Finite sets})
A definable set $A$ is finite if and only if $\dim(A)=0$.

\smallskip

We also assume that $G$ satisfies the \emph{dcc}. In particular, $G$ has a smallest definable subgroup $G\oo$ of finite index, the intersection of all of them. Then:

\medskip
\emph{Let $G$ be a solvable group equipped with a dimension and with dcc. Let $A$ and
$B$ be two connected definable subgroups of $G$ which normalize each other. Then the subgroup $[A,B]$ is definable and connected.}

\medskip
Indeed, applying the reductions in the proof of  \cite[Thm.6.1]{BJO}, i.e. the claims 1,2 and 3 in the proof of Theorem \ref{MainTheoConnected} above, it  suffices to handle the following problem: given a connected solvable group $G$ equipped with a dimension and with dcc and such that any proper  normal connected definable subgroup is central, its derived subgroup $G'$ is definable and connected. Now, if $G$ is abelian then $G'$ is trivial and we are done. If $G$ is not abelian, then an argument in \cite[Thm.2.12]{PPS00-S} shows that there exists a proper normal connected definable subgroup $C$ of $G$ such that $G/C$ is abelian. This ends the proof since $G'\leq C \leq Z(G)$ and thus $G'$ is definable by the corresponding version of Fact \ref{rudimentary}.

We recall the argument in \cite[Thm.2.12]{PPS00-S}. Take $C$ a proper normal connected subgroup $C$ of $G$ of
maximal dimension. Suppose that $H:=G/C$ is not abelian. Then we prove that there exists a proper normal definable subgroup
$\widetilde{C}$ of $G$ such that $C$ is a subgroup of $\widetilde{C}$ of finite index and
$G/\tilde{C}$ is abelian. This yields a contradiction because $(G/C)'$ would be finite and therefore by \cite[Fact 3.1]{BJO} we would obtain that $G/C$ is abelian (just consider for each $g\in G$ the action by conjugation of $G$ over the finite set $g^G$). Since $H$ is a non-abelian solvable group, there exists $n\in
\mathbb{N}$, $n>1$, such that
$$1=H^{(n)}< H^{(n-1)}<\cdots<H$$
where $H^{(k)}:=[H^{(k-1)}]'$ for each $k\in \N$.
Let $m\in \mathbb{N}$, $m>1$, be minimal with $H^{(m)}$ finite. Let $\widetilde{C}$ be 
the normal definable subgroup of $G$ such that $\widetilde{C}/C=H^{(m)}$ and consider $H_1:=G/\widetilde{C}\simeq H/H^{(m)}$, where the symbol $\simeq$ means that the groups are definably isomorphic. Note that $E:=Z(C_{H_1}(H_1^{(m-1)}))$ contains $H_1^{(m-1)}$. Since $E$ is an infinite abelian normal definable subgroup of $H_1$, by maximality of $\dim(C)$  we get that $E=H_1$, so that $H_1$ is abelian, as required.
\end{remark}

\section{Malcev's cross-section}\label{s:malcev}
As an application of our previous results, we prove an o-minimal version of Fact \ref{Malcev}. We need first to study Levi decompositions in the simply-connected case. The following lemma follows from \cite{CP}, for the sake of completeness we provide a proof which becomes somewhat easier in our particular setting.

\begin{lemma}\label{levi} Let $G$ be a simply-connected definable group. Then there exists a semisimple simply-connected definable subgroup $S$ of $G$ such that $G=R(G)S$ and $R(G)\cap S=1$. 
\end{lemma}
\begin{proof}It is enough to prove there is a connected semisimple definable subgroup $S$ such that 
$G=R(G)S$ and $R(G)\cap S$ is finite. For, in that case the quotient
$$G/R(G)=R(G)S/R(G)\simeq S/(R(G)\cap S)$$
is simply-connected by Proposition \ref{propsimplycom}. Then, by Fact \ref{nocover}, the finite normal subgroup $R(G)\cap S$ of $S$ must be trivial. In particular, $S\simeq G/R(G)$ is simply-connected, as required.

Suppose first that $G\leq \GL(n,R)$ is linear. Let $\g=\mathfrak{r}+\mathfrak{s}$ be a Levi decomposition of the Lie algebra $\g$ of $G$, where $\mathfrak{r}$ denotes the radical of $\g$. We note that $\text{Lie}(R(G))=\mathfrak{r}$. Indeed, since $R(G)$ is solvable, its Lie algebra $\text{Lie}(R(G))$ is solvable \cite[Lem.3.7]{BBO} and therefore $\text{Lie}(R(G))\subseteq \mathfrak{r}$. In particular, since $G/R(G)$ is semisimple, it follows from Fact \ref{semisimple} that $\g/\text{Lie}(R(G))$ is semisimple and so $\text{Lie}(R(G))=\mathfrak{r}$. On the other hand, since $\mathfrak{s}=[\mathfrak{s},\mathfrak{s}]=[a(\mathfrak{s}),a(\mathfrak{s})]$ is algebraic, there is an algebraic subgroup $S_1$ of  $\GL(n,R)$ whose Lie algebra is $\mathfrak{s}$. Therefore $S:=S_1\oo$ is a  connected semisimple definable subgroup of $G$ such that $G=R(G)S$ and $R(G)\cap S$ is finite, as desired.

Now, suppose that $G$ is almost-linear, i.e. there is a finite normal (central) subgroup $N$ of $G$ such that $G/N$ is linear. Let $\pi:G\rightarrow G/N$ be the canonical projection, and let $R_1:=R(G)N/N$ be the radical of $G/N$. By the above there exists a connected semisimple definable subgroup $S_1$ of $G/N$ such that $G/N=R_1S_1$ and $R_1\cap S_1=1$. Then for the connected semisimple definable subgroup $S:=\pi^{-1}(S_1)\oo$ of $G$ we have that $G=R(G)S$ and $R(G)\cap S$ is finite, as required.

For the general case, recall that the kernel of the definable homomorphism ${Ad:G\rightarrow \text{Aut}(\g)}$ is $Z(G)$, and so the definable group $G/Z(G)\oo$ is almost-linear and simply-connected by Proposition \ref{propsimplycom}. Denote $\pi:G\rightarrow G/Z(G)\oo$ the canonical projection and let $R_1:=R(G/Z(G)\oo)=R(G)/Z(G)\oo$. Let $S_1$ be a simply-connected semisimple definable subgroup of $G/Z(G)\oo$ such that $G/Z(G)\oo=R_1S_1$ and $R_1\cap S_1=1$. 

Consider the connected definable subgroup $B:=\pi^{-1}(S_1)$ of $G$. By Proposition \ref{propsimplycom}, since both $S_1=B/Z(G)\oo$ and $Z(G)\oo$ are simply-connected, we have that $B$ is simply-connected. In particular, $S:=[B,B]$ is also definable and simply-connected by Theorem \ref{MainTheoConnected}. 
Let us show that $Z(G)\oo\cap S$ is finite. To be semisimple and/or simply-connected is preserved under elementary extensions. Therefore, we can assume that $\mathcal{R}$ is $\aleph_1$-saturated. Since the simply-connected definable group $B$ is a central extension of the semisimple group $S_1$, for each $n\in \N$
$$Z(G)\oo\cap [B,B]_n \subseteq Z(B)\oo\cap [B,B]_n$$
is finite by Fact \ref{f:HPP}. In particular, since $[B,B]$ is definable by Proposition \ref{propsimplycom}, by saturation there exists $n_0\in \N$ such that
$$Z(G)\oo\cap [B,B]_{n_0}=Z(G)\oo\cap [B,B]_{n}$$
for all $n\geq n_0$ and therefore $Z(G)\oo \cap S=Z(G)\oo\cap [B,B]_{n_0}$ is finite.

On the other hand, by Fact \ref{semisimple} we have that $S'_1=S_1$. Thus, $\pi(S)=S_1$ and 
$$S_1 \simeq SZ(G)\oo/Z(G)\oo\simeq S/(Z(G)\oo\cap S).$$
Since $S_1$ is simply-connected, it follows from Fact \ref{nocover} that  $Z(G)\oo \cap S$ is trivial. In particular, $S$ is a simply-connected semisimple definable subgroup of $G$.  Moreover, we clearly have that $G=R(G)S$ and
$(R(G)\cap S)\oo \trianglelefteq R(S)=1$, as required.
\end{proof}

In \cite{PS05} the authors show that if $G$ is a connected definable group and $H$ is a contractible normal definable subgroup of $G$ then there is a continuous definable \emph{cross-section} (or just \emph{section}) of the projection map $\pi:G\rightarrow G/H$, i.e., a continuous definable map $\sigma:G/H\rightarrow G$ such that $\pi\circ \sigma=\text{id}$. The other classical result concerning the existence of cross-sections is, in the o-minimal setting, an easy consequence of the results in \cite{BMO}:

\begin{lemma}\label{contractible}Let $G$ be a connected definable group and let $H$ be a normal connected definable subgroup of $G$. If $G/H$ is contractible then there exists a continuous definable section of the projection map $\pi:G\rightarrow G/H$.
\end{lemma}
\begin{proof} By \cite[Cor.2.4]{BMO} the projection map $G\rightarrow G/H$ is a definable fibration. Since $G/H$ is contractible, there exists a continuous definable map $F:G/H\times [0,1] \rightarrow G/H$ such that $F(-,0)=\text{id}$ and $F(-,1)$ is the constant function $c:G/H\rightarrow G/H:\bar{g}\mapsto \bar{1}$. Consider the lifting $\widetilde{c}:G/H\rightarrow G:\bar{g}\mapsto 1$ of $c$. By the homotopy lifting property of the projection $\pi$ with respect to all definable sets, there is a continuous definable map
$$\widetilde{F}:G/H\times [0,1]\rightarrow G$$
such that $\pi \circ \widetilde{F}=F$ and $\widetilde{F}(-,1)=\widetilde{c}$. In particular, the continuous definable map $\sigma:=\widetilde{F}(-,0):G/H\rightarrow G$ satisfies $\pi \circ \sigma =\text{id}$, as desired.
\end{proof}

We already have all the ingredients to prove the existence of cross-sections in the simply-connected case:

\begin{theorem}\label{Cor2}Let $G$ be a simply-connected definable group and let $H\trianglelefteq G$ be a connected definable subgroup. Then there exists a continuous definable section of the projection map $\pi:G\rightarrow G/H$.
\end{theorem}

\begin{proof}We prove it by induction on $\dim(G)$. The initial case $\dim(G)=0$ is obvious, so we assume that $\dim(G)\geq 1$ and the statement holds for all simply-connected definable groups of dimension less than $\dim(G)$. Let $H\trianglelefteq G$ be a connected definable subgroup and $\pi:G\rightarrow G/H$ the canonical projection. Note that both $H$ and $G/H$ are simply-connected by Proposition \ref{propsimplycom}.

\smallskip
\noindent\emph{Claim. If there are proper connected definable subgroups $A_1$ and $B_1$ of $G/H$ such that $A_1B_1=G/H$ and $A_1\cap B_1=1$ then there exists a continuous definable section of $\pi:G\rightarrow G/H$.}

\begin{proof}The map
$$\phi:A_1 \times B_1 \rightarrow  G/H:(a,b)\mapsto ab$$
is clearly a definable homeomorphism. In particular, since $\pi_1(A_1\times B_1)=\pi_1(A_1)\times \pi_1(B_1)$ by \cite[Lem.2.2]{EO} and $\phi_*: \pi_1(G/H)\rightarrow \pi_1(A_1\times B_1)$ is an isomorphism, it follows that both $A_1$ and $B_1$ are simply-connected. Moreover, by Proposition \ref{propsimplycom} the proper definable subgroups $A:=\pi^{-1}(A_1)$ and $B:=\pi^{-1}(B_1)$ of $G$ are simply-connected. 

By induction there are continuous definable maps
$$\sigma_A: A_1\rightarrow A \qquad \& \qquad \sigma_B: B_1\rightarrow B$$
such that $\pi \circ\sigma_A= \text{id}$ and $\pi \circ\sigma_B=\text{id}$. Consider the continuous definable maps
$$\sigma_{A\times B}:A_1\times B_1 \rightarrow A \times B:(x,y)\mapsto (\sigma_A(x),\sigma_B(y))$$
and
$$\psi:A\times B \rightarrow G:(x,y)\mapsto xy.$$
Then $$\sigma:=\psi \circ \sigma_{A \times B} \circ \phi^{-1}:G/H\rightarrow G$$
is a continuous definable map which satisfies $\pi\circ \sigma=\text{id}$, as desired.\end{proof}

By Proposition \ref{levi} there exists a definable simply-connected subgroup $S_1$ of $G/H$ such that $G/H=R_1S_1$ and $R_1\cap S_1=1$, where $R_1=R(G/H)$. Thus, without loss of generality either $G/H=R_1$ or $G/H=S_1$, by the Claim above. If $G/H=R_1$ then by Corollary \ref{Cor1} and Lemma \ref{contractible} there exists a continuous definable section $\sigma:G/H\rightarrow G$, so we can assume that $G/H=S_1$ is semisimple. Moreover, by Fact \ref{semisimple} there are
normal definable subgroups $C_1,\ldots,C_\ell$ of $S_1$ containing the finite center $Z(S_1)$ such that
$$S_1/Z(S_1)\simeq [C_1/Z(S_1)]\times \cdots \times [C_\ell/Z(S_1)]$$
and where $C_i/Z(S_1)$ is definably simple for $i=1,\ldots,\ell$. Since $C\oo_iZ(S_1)/Z(S_1)$ is a non-trivial normal definable subgroup of $C_i/Z(S_1)$, we have $C\oo_iZ(S_1)=C_i$ for each $i=1,\ldots,\ell$. Moreover, for each $i=1,\dots,\ell-1$ the intersection
$$(C\oo_1\cdots C\oo_i)\cap C\oo_{i+1} \subseteq (C_1\cdots C_i)\cap C_{i+1} \subseteq Z(S_1)$$
is finite. In particular, it follows that $\dim(C_1\oo\ldots C\oo_\ell)=\dim(C\oo_1)+\cdots+\dim(C\oo_\ell)=\dim(S_1)$ and since $S_1$ is connected, 
we have that $S_1:=C_1\oo\ldots C\oo_\ell$. If $\ell>1$ then we define $N:=C\oo_1 \cdots C\oo_{\ell-1}$, so that $S_1=NC\oo_\ell$ and $N\cap C\oo_\ell$ is finite. Since
$$C\oo_\ell/(N\cap C\oo_\ell) \simeq C\oo_\ell N/N=S_1/N$$
is simply-connected, we deduce from Fact \ref{nocover} that $N\cap C\oo_\ell=1$. Hence, again by the Claim we can assume that $G/H=S_1=C\oo_\ell$, and so every  proper normal definable subgroup of $G/H$ is finite and central.

Next, suppose that $R(H)$ is not trivial. Since $H$ is normal in $G$, and  $R(H)$ is characteristic in $H$, we get that $R(H)$ is normal in $G$. Thus, by Proposition \ref{propsimplycom} the connected definable group $G/R(H)$ is simply-connected. By induction there is a continuous definable section $\sigma_0:G/H \rightarrow G/R(H)$ of the projection map $G/R(H)\rightarrow G/H$. On the other hand, since $R(H)$ is solvable and simply-connected, by Corollary \ref{Cor1} the group $R(H)$ is contractible. Thus, by \cite[Thm.5.1]{PS05} we also have a continuous definable section $\sigma_1:G/R(H)\rightarrow G$ of the projection $G\rightarrow G/R(H)$. In particular, $\sigma:=\sigma_1\circ \sigma_0$  is the desired section of $\pi:G\rightarrow G/H$. Hence, we can assume that $H$ is semisimple.

Finally, by our assumptions on $H$ and $G/H=S_1=C\oo_\ell$, we can assume that $G$ is semisimple. In particular, there are normal definable subgroups $E_1,\ldots,E_m$ of $G$ containing the finite center $Z(G)$ such that  $$G/Z(G)\simeq [E_1/Z(G)]\times \cdots \times [E_m/Z(G)],$$
and where $E_i/Z(G)$ is definably simple for $i=1,\ldots,m$. Each $E\oo_iH/H$ is a normal definable subgroup of $G/H$ and therefore either it equals $G/H$ or it is finite. Thus, we can assume that $E\oo_1H=G$ and so the definable homomorphism
$$\pi \upharpoonright _{E\oo_1}:E\oo_1\rightarrow G/H$$
is surjective. Since $\ker(\pi \upharpoonright _{E\oo_1})$ is a  normal definable subgroup of $E_1$, it is finite. It follows from Fact \ref{nocover} that  $\pi \upharpoonright _{E\oo_1}$ is a definable isomorphism. The inverse of $\pi \upharpoonright _{E\oo_1}$ gives the required  continuous definable section $\sigma: G/H\rightarrow G$.
\end{proof}

\end{document}